\documentclass[11pt]{elsarticle} 
\usepackage{setspace}
\journal{Discrete Optimization}

\usepackage[utf8]{inputenc} 

\usepackage{subfigure}
\usepackage[margin=1in]{geometry}
\usepackage{graphicx} 
\usepackage{epstopdf}

\usepackage{booktabs} 
\usepackage{array} 
\usepackage{paralist} 
\usepackage{verbatim} 
\usepackage{subfig} 

\usepackage{fancyhdr} 
\usepackage[nottoc,notlof,notlot]{tocbibind} 
\usepackage[titles,subfigure]{tocloft} 

\usepackage{amsmath, amsthm, amssymb}

\usepackage{multirow}
\usepackage{algorithm}
\usepackage{algorithmic}

\usepackage{amsthm}

\newtheorem{lemma}{Lemma}

\newtheorem{proposition}{Proposition}

\newtheorem{fact}{Fact}

\allowdisplaybreaks

\DeclareMathOperator*{\arcsec}{arcsec}
\DeclareMathOperator*{\argmin}{argmin}
\DeclareMathOperator*{\argmax}{argmax}

\usepackage{rotating}
\usepackage{pdflscape}
\usepackage{pgfplots}
\usepgfplotslibrary{units}
\usepackage{url}
\usepackage{color}

\graphicspath{{figures/}} 

\allowdisplaybreaks

\usepackage{color}
\usepackage{bm}

\newcommand{\cS}{\mathcal{S}}

\newcommand{\cT}{\mathcal{T}}

\newcommand{\cL}{\mathcal{L}}

\newcommand{\cX}{\mathcal{X}}
\newcommand{\cQ}{\mathcal{Q}}

\def\blue{}

\begin{document}

\begin{frontmatter}
\title{Rational Polyhedral Outer-Approximations  of the Second-Order Cone
}

\author{Burak Kocuk} 
\address{Industrial Engineering Program,  Sabanc{\i} University, Istanbul, Turkey 34956\\burak.kocuk@sabanciuniv.edu}

\begin{abstract}
It is well-known that the second-order cone can be outer-approximated to an arbitrary accuracy~$\epsilon$ by a polyhedral cone of compact size defined by irrational data. In this paper, we propose two rational polyhedral outer-approximations of compact size retaining the same guaranteed accuracy~$\epsilon$. The first outer-approximation has the same size as the optimal but irrational outer-approximation from the literature.  In this case, we  provide a practical approach to obtain such an approximation defined by the smallest integer coefficients possible, which requires solving a few, small-size  integer quadratic programs. The second outer-approximation has a size larger than  the optimal irrational outer-approximation by a linear additive factor in the dimension of the second-order cone. However, in this case, the construction is explicit, and it is possible to derive an upper bound on the largest coefficient, which is sublinear in $\epsilon$ and logarithmic in the dimension. We also propose a third outer-approximation, which yields the best possible approximation accuracy given an upper bound on the size of its coefficients. Finally, we discuss two {theoretical} applications in which having a rational polyhedral outer-approximation  is crucial, {and run some  experiments which explore the benefits of the formulations proposed in this paper from a computational perspective.}
\end{abstract}


\end{frontmatter}

\section{Introduction}

In their classic paper \cite{ben2001polyhedral}, Ben-Tal and Nemirovski have shown that the second-order cone  in dimension  $N$, defined as  
\[ 
L^{N} := \big \{ x \in \mathbb{R}^{N} : \  \sqrt{  x_{1}^2 + \cdots + x_{N-1}^2 } \le x_N \big\}, 
\]
can be outer-approximated to an arbitrary accuracy $\epsilon$ by a polyhedral cone in an extended space. In their construction, the required number of additional variables and linear inequalities grows polynomially  in $N$ and $\log \epsilon^{-1}$, rendering a very efficient way of approximating $L^{N}$ by a polyhedral set with a compact size representation. Moreover, they also prove that their construction is the smallest possible in size (up to a constant factor). 
For all practical purposes, this means that any second-order cone program (SOCP) can be well-approximated by a linear program (LP) of reasonable size.

Although directly solving an SOCP via a primal-dual interior point method is typically more efficient than solving an approximating LP \cite{glineur2000computational}, the situation changes considerably in the presence of integer variables. Polyhedral approximations are still very powerful for solving mixed-integer SOCPs as the benefits of warm-starting for LPs can be utilized throughout the branch-and-bound algorithm \cite{vielma2008lifted, barmann2016polyhedral, VINEL201766, lubin2018polyhedral}. 
This  is particularly important for problems that can be formulated as MISOCPs; such problems arise in different applications including location and inventory management \cite{Atamturk12}, 
power distribution systems \cite{Kocuk2015}, options pricing  \cite{Pinar13}, and 
Euclidean $k$-center problems \cite{Brandenberg}.

Despite the fact that Ben-Tal and Nemirovski's polyhedral approximation of the second-order cone is  \textit{optimal} in terms of the size of the formulation 
 given an arbitrary precision $\epsilon >0$, it is defined by a set of linear inequalities given by some irrational coefficients. This might be an issue in terms of the computational solution procedures due to the accumulation of  rounding  errors. Also, the rich theory of mixed-integer programs defined by rational polyhedra
  cannot be applied to the outer-approximating polyhedron, which, otherwise, could have been useful to understand some theoretical properties of MISOCPs.

In this paper, we provide {three} compact size, rational polyhedral outer-approximations of the second-order cone $L^N$ with slightly different properties. 
\begin{itemize}
\item
In the first case, discussed in Section \ref{sec:rationalOuter}, we construct a rational polyhedral outer-approximation of $L^N$  given an arbitrary accuracy $\epsilon > 0$ whose size in terms of the number of variables and linear inequalities is the same as the size of Ben-Tal and Nemirovski's optimal 
 construction. 
Moreover, we  provide a simple algorithm which produces the smallest integers possible in the inequality description of the outer-approximating polyhedron. These integer coefficients are obtained by solving a limited number of small-size  integer quadratic programs.

\item
In the second case, discussed in Section \ref{sec:rationalOuterHeur}, we  construct a rational polyhedral outer-approximation of $L^N$  given an arbitrary accuracy $\epsilon > 0$, whose size  in terms of the number of variables and linear inequalities is  larger by a linear additive factor of~$N$ than the size of Ben-Tal and Nemirovski's optimal construction. One advantage of the second construction is that the  coefficients of the outer-approximating polyhedron are obtained in closed form. This  allows us to derive an upper bound on the largest integer used in the formulation, which grows  {sublinearly in $\epsilon$ and logarithmically in $N$.}

\item 
{In the third case, discussed in Section \ref{sec:deltaGivenM}, we construct a rational polyhedral outer-approximation of $L^N$  given an upper bound $C$ on the coefficients of its inequality description. We show that our construction yields the best possible approximation accuracy $\epsilon$ given $C$. Moreover, its  size  in terms of the number of variables and linear inequalities is   larger by a linear additive factor of~$N$, in the worst case, than the size of Ben-Tal and Nemirovski's optimal construction with the same accuracy. 
}
\end{itemize}
Finally, we conclude our paper in Section~\ref{sec:applications} by discussing some applications of rational outer-approximations of the second-order cone. 
In particular, we focus on the analysis of a set defined as the integral vectors in the intersection of multiple balls, and explore how to estimate the integrality gap of optimizing a linear function over this set and how to compute its cardinality. In this analysis, we leverage the theory of rational polyhedra in mixed-integer linear programming and our rational polyhedral approximations of the second-order conic sets. 
{We also present some computational experiments which involve minimizing a linear function over the intersection of balls. We empirically observe that   it might be more efficient to use the outer-approximations defined with integer coefficients proposed in this paper compared to Ben-Tal and Nemirovski's   outer-approximation defined with irrational coefficients.
}

\section{First Outer-Approximation: An Optimized Construction {Given the Desired Accuracy}}
\label{sec:rationalOuter}

In this section, our aim is to outer-approximate $L^N$ to arbitrary accuracy with a compact size polyhedral cone defined by rational data. 
We will first present our main result for $L^3$ in Section~\ref{sec:extendL3} and obtain an \textit{optimized} construction in 
Section~\ref{sec:optimizedConstruction}. Using the outer-approximation of $L^3$ as a building block, we utilize the \textit{tower-of-variables} construction to obtain a rational polyhedral outer-approximation for $L^N$ in Section~\ref{sec:extendLn}.

\subsection{Extending Ben-Tal and Nemirovski's Result for  $L^3$}
\label{sec:extendL3}

Let $\nu \in \mathbb{Z}_{++}$ and $\theta_j \in (0,\pi/2)$, $j=0,\dots,\nu$.
Let us define the following system in variables $(x_1,x_2,x_3,\xi^j,\eta^j)$, $j=0,\dots,\nu$:
\begin{subequations} \label{eq:outer}
\begin{align}
 \xi^0  &\ge |x_1|\label{eq:step0x1} \\
\eta^0  &\ge |x_2| \label{eq:step0x2} \\
\xi^j  &= \ \ \ \ \cos (\theta_j) \xi^{j-1} + \sin (\theta_j) \eta^{j-1} \ \quad j=1,\dots,\nu \label{eq:stepjx1} \\
\eta^j  &\ge |-\sin (\theta_j) \xi^{j-1} + \cos (\theta_j) \eta^{j-1}| \quad j=1,\dots,\nu \label{eq:stepjx2} \\
\xi^\nu & \le x_3 \label{eq:stepx3} \\ 
\eta^\nu & \le \tan (\theta_\nu) \xi^\nu \label{eq:stepx3_}
\end{align}
\end{subequations}

The following proposition and its proof is a  modification of Proposition 2.1 in \cite{ben2001polyhedral}.
\begin{proposition}\label{thm:outerL3general}
Let $\nu \in \mathbb{Z}_{++}$, $\kappa \in[1,2)$ and $\theta_j \in (0,\pi/2)$ for $j=1,\dots,\nu$. Assume that $\theta_j \in [\theta_{j-1}/2, \kappa\theta_{j-1}/2]$, $j=1,\dots,\nu$ with $\theta_0:= \pi/2$. 
Let us define 
\[
P^\nu (\theta):=\{ x\in\mathbb{R}^3: \ \exists (\xi^j,\eta^j)_{j=0}^\nu: \  \eqref{eq:outer} \}.
\]
Then, we have
\[
L^3 \subseteq P^\nu (\theta) \subseteq   {\sec(\theta_\nu)} L^3.
\]
Moreover, in order to obtain a $(1+\delta)$ outer-approximation of $L^3$ for $\delta>0$, one can choose
\begin{equation}\label{eq:approxGuarantee}
\nu ( \delta, \kappa) = \bigg\lceil  \frac{\log \frac{2\arcsec(1+\delta)}{\pi}}{\log \frac{\kappa}{2}} \bigg \rceil.
\end{equation}
\end{proposition}
\begin{proof}
We first prove that $L^3 \subseteq P^\nu (\theta) $. Let $x\in L^3$ and consider $(\xi^j,\eta^j)$, $j=0,\dots,n$ defined  as follows:
\[
(\xi^j,\eta^j) = 
\begin{cases}
(|x_1|, |x_2|) & \text{ if } j=0 \\
(\cos (\theta_j) \xi^{j-1} + \sin (\theta_j) \eta^{j-1}, |-\sin (\theta_j) \xi^{j-1} + \cos (\theta_j) \eta^{j-1}|) & \text{ if } j =1,\dots,\nu.
\end{cases}
\]
By construction, $(\xi^j,\eta^j)_{j=0}^\nu$ satisfy constraints \eqref{eq:step0x1}--\eqref{eq:stepjx2}. Moreover, due to the  choice of $\theta_j$'s, the angle of the vector with end point $(\xi^j,\eta^j)$ is at most $\theta_j$, for $j=1,\dots,\nu$.  Combining this with the fact that $x_3 \ge \|(x_1,x_2)\|_2 = \|(\xi^j,\eta^j)\|_2$, for $j=0,\dots,\nu$, we conclude that constraints \eqref{eq:stepx3}--\eqref{eq:stepx3_} are also satisfied. 

We next prove that $ P^\nu (\theta) \subseteq {\sec(\theta_\nu)} L^3$. Let  $(x_1,x_2,x_3,\xi^j,\eta^j)$, $j=0,\dots,\nu$ satisfy  system~\eqref{eq:outer}. Then, due to constraints \eqref{eq:step0x1}--\eqref{eq:stepjx2}, we have that 
\[
\|(x_1,x_2)\|_2 \le \|(\xi^0,\eta^0)\|_2 \le \|(\xi^1,\eta^1)\|_2 \le  \cdots \le \|(\xi^\nu,\eta^\nu)\|_2.
\]
Moreover, constraints \eqref{eq:stepx3}--\eqref{eq:stepx3_} imply that $\|(\xi^\nu,\eta^\nu)\| \le \sqrt{1+\tan^2 (\theta_\nu)} x_3$, proving that $\|(x_1,x_2)\|_2 \le {\sec (\theta_\nu)} x_3$.

Finally, due to the fact that $\theta_\nu \le \frac\pi2(\frac{\kappa}{2})^\nu$,  the smallest $\nu$ that guarantees the condition  ${\sec (\theta_\nu)}  \le 1+\delta$ is given in \eqref{eq:approxGuarantee}.
\end{proof}
It is proven in  \cite{ben2001polyhedral} that the smallest size (up to a constant factor) extended formulation to achieve a $(1+\delta)$-approximation of $L^3$ can be obtained by choosing
\begin{equation}\label{eq:nuOptimal}
\nu_\delta =  \left\lceil  {\log_2 \frac{\pi}{2\arcsec(1+\delta)}}  \right \rceil, \ \theta_j = \frac{\pi}{2^{j+1}}, j=1,\dots,\nu_\delta \text{ and } \kappa=1.
\end{equation}
However, under this set of choices, the polyhedral cone $P^{\nu_\delta} (\theta)$ given by  system \eqref{eq:outer} contains some irrational coefficients (e.g. $\cos(\theta_1) = 1/\sqrt{2}$). We will now show that a rational $(1+\delta)$-approximation with the same size can also be obtained. We start with a lemma.

\begin{lemma}\label{lem:rationalSineCosine}
Let $0 < \underline \theta < \overline \theta < \frac\pi2$. Then, there exists $\theta \in [ \underline \theta , \overline \theta ]$ such that $\cos(\theta) \in \mathbb{Q}$ and $\sin(\theta) \in \mathbb{Q}$.
\end{lemma}
\begin{proof}
We will prove this lemma using the fact that any integer Pythagorean triple can be written in the form
\[
[ m^2-n^2, \ 2mn, \ m^2+n^2],
\]
where $m$ and $n$ are positive integers with $m > n$. To find an angle $\theta$ with rational sine and cosine values, it suffices to find positive integers $m$ and $n$ such that
\begin{equation*}\label{eq:sinThetaBound}
\sin (\underline \theta) \le \frac{m^2-n^2}{m^2+n^2} \le \sin(\overline \theta).
\end{equation*}
By rearranging terms, we obtain
\begin{equation}\label{eq:sinThetaBoundRearranged}
\sqrt{ \frac{1+\sin (\underline \theta)}{1-\sin (\underline \theta)} } n \le m \le \sqrt{ \frac{1+\sin (\overline \theta)}{1-\sin (\overline \theta)} } n .
\end{equation}
Since $\underline \theta < \overline \theta $, one can show that $ L :=\sqrt{ \frac{1+\sin (\underline \theta)}{1-\sin (\underline \theta)} } < \sqrt{ \frac{1+\sin (\overline \theta)}{1-\sin (\overline \theta)} } =: U$. Finally, due to the fact that rational numbers are dense in the real numbers, we conclude that there exist positive integers $m$ and $n$ such that $\frac{m}{n} \in [L,U]$.
\end{proof}

We now state our main result.
\begin{proposition}\label{thm:outerL3rational}
For given $\delta>0$, let $\nu_\delta$   be chosen according to~\eqref{eq:nuOptimal} 
 and $\theta_0 = \pi/2$.
 Assume that $ {\log_2 \frac{\pi}{2\arcsec(1+\delta)}} \not\in \mathbb{Z}$.
Then, there exists a set of values $\theta_j$ with $\sin(\theta_j)\in\mathbb{Q}$ and $\cos(\theta_j)\in\mathbb{Q}$, $j=1,\dots,\nu$ such that a rational $(1+\delta)$-approximation $P^{\nu_\delta} (\theta)$ of $L^3$ can be obtained.
\end{proposition}
\begin{proof}
First of all, we choose the largest parameter $\kappa$   that   satisfies the condition $ \sec\big(\frac\pi2(\frac{\kappa}{2})^{\nu_\delta}\big) \le 1+\delta $, and denote it as  
\begin{equation}\label{eq:kappaOptimal}
\kappa_\delta = 2 \left( \frac{2}{\pi} \arcsec(1+\delta) \right)^{1/\nu_\delta}.
\end{equation}
Note that   $\kappa_\delta > 1$ since we assume that  $ {\log_2 \frac{\pi}{2\arcsec(1+\delta)}} \not\in \mathbb{Z}$.

Due to Lemma  \ref{lem:rationalSineCosine},  there exists an angle $\theta_j$ with $\sin(\theta_j), \cos(\theta_j)\in\mathbb{Q}$  such that $\theta_j \in [\theta_{j-1}/2, {\kappa_\delta}\theta_{j-1}/2]$, $j=1,\dots,\nu_\delta$. Therefore, the approximation $P^{\nu_\delta} (\theta)$  is a rational polyhedral cone. Finally, since
$\kappa_\delta$ and $\nu_\delta$ satisfy the condition that $ \sec\big(\frac\pi2(\frac{\kappa_\delta}{2})^{\nu_\delta}\big) \le 1+\delta $, we achieve a $(1+\delta)$-approximation of the cone $L^3$ due to Proposition~\ref{thm:outerL3general}.
%
%
%
\end{proof}

\subsection{An Optimized Outer-Approximation for  $L^3$}
\label{sec:optimizedConstruction}

In this section, we will first reformulate the outer-approximation  $P^\nu (\theta)$ as a polyhedral set defined by integral data. Then, we will propose a way to obtain the {smallest} integer coefficients that can be used in this reformulation.

For $\theta_j \in (0,\pi/2)$ with rational sine and cosine, we can write 
\[
\sin(\theta_j) = \frac{a_j}{c_j} \text{ and } \cos(\theta_j) = \frac{b_j}{c_j},
\]
for some positive coprime integers $a_j, b_j, c_j$. 
By rewriting  equations \eqref{eq:stepjx1}, \eqref{eq:stepjx2} and \eqref{eq:stepx3_}  respectively  as 
\begin{subequations} \label{eq:outerReform}
\begin{align}
c_j \xi^j  &= \ \ \ \ \  b_j \xi^{j-1} + a_j \eta^{j-1} \ \quad j=1,\dots,\nu \label{eq:stepjx1R} \\
c_j \eta^j  &\ge |-a_j \xi^{j-1} + b_j \eta^{j-1}| \quad j=1,\dots,\nu \label{eq:stepjx2R} \\
b_\nu \eta^\nu & \le  a_\nu \xi^\nu \label{eq:stepx3_R},
\end{align}
\end{subequations}
we reformulate the polyhedral set  $P^\nu (\theta)$ as
\begin{equation}\label{eq:outerInteger}
\tilde P^\nu (\theta):=\{ x\in\mathbb{R}^3: \ \exists (\xi^j,\eta^j)_{j=0}^\nu: \  \eqref{eq:step0x1}-\eqref{eq:step0x2}, \eqref{eq:stepx3}, \eqref{eq:outerReform} \}.
\end{equation}
Note that the polyhedral set $\tilde P^\nu (\theta)$ is now defined by integral data.

Although Proposition \ref{thm:outerL3rational} implies that the outer-approximation $\tilde P^{\nu_\delta} (\theta)$ can be defined by integral data, it does not give an explicit construction. In the remainder of this section, we provide an \textit{optimized} construction in the sense that the integers used in the formulation are as small as possible. This is a desired property since it is known that the computational solution methods for optimization problems with very large coefficients are subject to numerical issues. 

The exact procedure to obtain the coefficients is summarized in Algorithm \ref{alg:optConst}. In each step of the algorithm, we minimize the hypotenuse of the next Pythagorean triple, which seems a reasonable approach to obtain a formulation with small coefficients. We use the Python programming language to implement Algorithm \ref{alg:optConst} and Gurobi  to solve the integer quadratic programs (we tighten the default numerical tolerances of the solver to obtain accurate results).
\begin{algorithm}
\caption{Optimized construction of $\tilde P^{\nu_\delta} (\theta)$ given $\delta > 0$.}
\label{alg:optConst}
\begin{algorithmic}
\STATE Compute $\nu_\delta$ according to \eqref{eq:nuOptimal}  and choose $\kappa \in (1, \kappa_\delta]$, where $ \kappa_\delta$ is defined as in \eqref{eq:kappaOptimal}.
\STATE Initialize $\theta_0 = \pi/2$.
\FOR{$j =1, \dots, \nu_\delta$}
\STATE Set $\underline \theta = \theta_{j-1}$ and $\overline \theta = \kappa \theta$.
\STATE Obtain
\[
(m_s, n_s) := \argmin_{ m\in \mathbb{Z},n \in \mathbb{Z} }\{ m^2+n^2 : \ \eqref{eq:sinThetaBoundRearranged}, n \ge 1 \}
\text{ and }
(m_c, n_c) := \argmin_{ m\in \mathbb{Z},n \in \mathbb{Z} }\{ m^2+n^2 : \ \eqref{eq:cosThetaBoundRearranged}, n \ge 1 \}.
\]
where
\begin{equation}\label{eq:cosThetaBoundRearranged}
\sqrt{ \frac{1+\cos (\overline \theta)}{1-\cos (\overline \theta)} } n \le m \le \sqrt{ \frac{1+\cos (\underline \theta)}{1-\cos (\underline \theta)} } n .
\end{equation}

\IF{ $m_s^2+ n_s^2 < m_c^2+ n_c^2$ }
\STATE Set $ a_j = m_s^2 - n_s^2 $, $b_j = 2m_s n_s$, $c_j = m_s^2 + n_s^2$.
\ELSE
\STATE Set $ a_j =  2 m_c n_c$, $b_j =m_c^2 - n_c^2 $, $c_j = m_c^2 + n_c^2$.
\ENDIF

\STATE
Set $\theta_j= \arcsec(c_j/b_j)$. 

\ENDFOR

\end{algorithmic}
\end{algorithm}

Some sample results are provided in Table \ref{eq:optConstruct-4-6} for $\delta=10^{-5}, 10^{-6}, 10^{-7} $. We remark that that the optimal coefficients for the same index $j$ might be different depending on the value of $\delta $. Also, it is interesting to observe  that all the Pythagorean triples for $j\ge4$ are one of the following two forms for some positive integer $h$:
\begin{equation}\label{eq:twoTypes}
i) \ [2h, h^2-1, h^2+1], \quad ii) \ [2h-1,2h^2-2h, 2h^2-2h+1].
\end{equation}
This observation will be exploited in Section \ref{sec:heuristicConstruction}, where we provide a different approximation scheme to $L^3$ with explicitly chosen Pythagorean triples.

\begin{table}[H]\footnotesize
\caption{Optimal constructions of the polyhedral cone $\tilde P^{\nu_\delta} (\theta)$ for different $\delta$ values. We have chosen $\kappa= (1-10^{-6})\kappa_\delta$, where $\kappa_\delta$ is defined as in \eqref{eq:kappaOptimal}, to circumvent the possible numerical precision issues.}\label{eq:optConstruct-4-6}
\begin{tabular}{c|cc|cc|cc}
\toprule
           & \multicolumn{ 2}{c|}{$\delta=10^{-5}$ ($\nu_\delta=9$)} & \multicolumn{ 2}{c|}{$\delta=10^{-6}$ ($\nu_\delta=11$)} & \multicolumn{ 2}{c}{$\delta=10^{-7}$ ($\nu_\delta=12$)} \\

      $j $ &  $\sec(\theta_j)$ &   $a_j,b_j,c_j$ & $\sec(\theta_j)$ &   $a_j,b_j,c_j$ & $\sec(\theta_j)$ &   $a_j,b_j,c_j$ \\
\midrule
         1 &   1.450000 &   21,20,29 &  1.4500000 &   21,20,29 & 1.42016807 & 120,119,169 \\

         2 &   1.096257 & 84,187,205 &  1.0962567 & 84,187,205 & 1.08333333 &    5,12,13 \\

         3 &   1.023226 & 168,775,793 &  1.0250000 &    9,40,41 & 1.02020202 &  20,99,101 \\

         4 &   1.006192 & 36,323,325 &  1.0061920 & 36,323,325 & 1.00501253 & 40,399,401 \\

         5 &   1.001634 & 35,612,613 &  1.0016340 & 35,612,613 & 1.00125078 & 80,1599,1601 \\

         6 &   1.000420 & 69,2380,2381 &  1.0004456 & 67,2244,2245 & 1.00032051 & 79,3120,3121 \\

         7 &   1.000113 & 133,8844,8845 &  1.0001240 & 127,8064,8065 & 1.00008114 & 157,12324,12325 \\

         8 &   1.000030 & 257,33024,33025 &  1.0000344 & 241,29040,29041 & 1.00002068 & 311,48360,48361 \\

         9 &   1.000008 & 493,121524,121525 &  1.0000096 & 457,104424,104425 & 1.00000529 & 615,189112,189113 \\

        10 &            &            &  1.0000027 & 865,374112,374113 & 1.00000135 & 1215,738112,738113 \\

        11 &            &            &  1.0000007 & 1637,1339884,1339885 & 1.00000035 & 2401,2882400,2882401 \\

        12 &            &            &            &            & 1.00000009 & 4741,11238540,11238541 \\

\bottomrule
\end{tabular}  
\end{table}

%
%
%
%
%
%
%
%
%
%
%
%
%

\subsection{Rational Outer-Approximation of  $L^N$}
\label{sec:extendLn}

In this section, we will use the rational outer-approximation of $L^3$ as a building block to obtain a rational outer-approximation of $L^N$. We remark that  the  \textit{tower-of-variables} construction used in this section has originally appeared in \cite{ben2001polyhedral} and we only present it here  in order to make the present paper self-contained.
 For convenience, we will assume  that $N=2^K+1$ for some $K\in\mathbb{Z}_{++}$  (otherwise, we can include additional variables which are equal to zero).

Let us define
\begin{subequations}\label{eq:tower}
\begin{align}
& x_{i} = y_{0,i}, x_N=y_{K,1}  &i& = 1,\dots, N-1 \label{eq:towerLinear} \\
&   (y_{k,2i-1} , y_{k,2i}, y_{k+1,i}) \in L^3  \ &k&=0,\dots,K-1, i=1,\dots, 2^{K-k-1}. \label{eq:towerCone}
\end{align}
\end{subequations}
For $L^N$, an extended formulation  can be given as
$
L^N = \{ x\in\mathbb{R}^{N} : \exists y_{k,i} : \eqref{eq:tower}\},
$
and for  $\delta>0$, an outer-approximation can be obtained as
\[
Q^N (\theta) := \{x\in\mathbb{R}^{n} : \exists y_{k,i} : \eqref{eq:towerLinear} , (y_{k,2i-1}, y_{k,2i}, y_{k+1,i}) \in  \tilde P^{\nu_\delta}(\theta)  \ k=0,\dots,K-1, i=1,\dots, 2^{K-k-1} \},
\]
where $ \tilde P^{\nu_\delta}$ is defined as in Section \ref{sec:optimizedConstruction}.
It is straightforward to show that 
\[
L^N \subseteq Q^N(\theta)  \subseteq (1+\delta)^K L^N.
\]
Finally, to achieve a $(1+\epsilon)$-approximation of the cone $L^N$, we can construct a $(1+\delta)$-approximation for $L^3$ such that
\begin{equation}\label{eq:deltaEpsilon}
1+\epsilon \ge (1+\delta)^K \iff \delta\le (1+\epsilon)^{1/K} - 1.
\end{equation}
Let us denote the largest $\delta$ value satisfying this relation as $\delta_\epsilon$. 
Then, $Q^N(\theta) $ would require $(2\nu_{\delta_\epsilon}+3)(N-2)$ many additional variables and  $(3\nu_{\delta_\epsilon}+6)(N-2)$ many constraints in its description (see system \eqref{eq:outer}), where 
\[
\nu_{\delta_\epsilon} =  \left\lceil  {\log_2 \frac{\pi}{2\arcsec\big((1+\epsilon)^{1/K}\big)}}  \right \rceil ,
\]
is computed according to \eqref{eq:nuOptimal}.

We remark that the  {tower-of-variables} construction is not the only way to obtain an extended formulation for $L^N$ that uses $L^3$ as a building block. In fact, one can also use a \textit{disaggregation} approach (see \cite{ben2001lectures, vielma2017extended}). Let us define the system
\begin{subequations}\label{eq:disaggreg}
\begin{align}
&   x_i^2 \le z_i x_N   \ &i&=1,\dots, N-1 \label{eq:disaggregCone} \\
&  \sum_{i=1}^{N-1} z_{i} \le x_N  \label{eq:disaggregLinear} \\
& x_N \ge 0,   z_i \ge 0   \ &i&=1,\dots, N-1 \label{eq:disaggregNonneg},
\end{align}
\end{subequations}
and obtain another extended formulation  as 
$
L^N = \{ x\in\mathbb{R}^{N} : \exists z_{i} : \eqref{eq:disaggreg}\}.
$
Note that  the rotated cone constraints \eqref{eq:disaggregCone} are linear transformations of $L^3$, hence, admit efficient polyhedral outer-approximations.

%
%
%
%
%

\section{Second  Outer-Approximation: A Closed-Form Construction with Small-Size Coefficients  {Given the Desired Accuracy}}
\label{sec:rationalOuterHeur}



In this section, we present another rational outer-approximation of $L^N$ with a  different  set of characteristics. 
We will first present our closed-form construction for $L^3$  in 
Section \ref{sec:heuristicConstruction}
and provide an upper bound for the largest coefficient used in its description in
Section \ref{sec:boundCoeff}.
Then, using the tower-of-variables construction again, we will obtain a rational polyhedral outer-approximation for~$L^N$ with an explicit upper bound on its largest coefficient in
Section \ref{sec:extendLn'}.

\subsection{A Closed-Form Outer-Approximation for  $L^3$}
\label{sec:heuristicConstruction}


In this section, we will consider the Pythagorean triples defined as
\begin{equation}\label{eq:heurConstruction}
(\hat a_j, \hat b_j, \hat c_j) = \begin{cases}
(120, \ 119, \ 169) & \text{ if } j=1 \\
(2h_j-1,2h_j^2-2h_j, 2h_j^2-2h_j+1 )  \text{ with } h_j := 2^{j-2}+2 & \text{ if } j =2, \dots, \nu,
\end{cases}
\end{equation}
and the corresponding angles
\begin{equation}\label{eq:heurConstructionAngle}
\hat \theta_j := \arcsec   ( {\hat c_j}/{ \hat b_j }  ), \ j=1,\dots,\nu.
\end{equation}
Notice that $(\hat a_j, \hat b_j, \hat c_j) $ are of the  type (ii) according to~\eqref{eq:twoTypes}  for $j=2,\dots,\nu$, and we have the recursion $h_j = 2(h_{j-1}-1)$ for $j=3,\dots,\nu$. 

Our main result in this section is the following proposition: 
\begin{proposition}\label{thm:heurConstruction}
Let $\delta \in (0,\frac14)$ and $\hat \theta =\pi/2$.
Consider $\tilde P^{\hat \nu_\delta} ( \hat \theta)$ with 
\begin{equation} \label{eq:hatNuDelta}
\hat \nu_\delta = \left\lceil \log_2 \left (-6+2\sqrt{1+2/\delta} \right) \right \rceil,
\end{equation}
and $\hat \theta_j$, $j=1,\dots, \hat \nu_\delta$ as defined in \eqref{eq:heurConstructionAngle}. 
Then, we have
\[
L^3 \subseteq \tilde P^{\hat \nu_\delta} ( \hat \theta) \subseteq   (1+\delta) L^3.
\]
Moreover, 
\[
\hat \nu_\delta  \le  \nu_\delta + 2.
\]
\end{proposition}
We will postpone the proof of Proposition \ref{thm:heurConstruction} to Section~\ref{sec:proofProp3} since we first need some preliminary results.

\subsubsection{Preliminary Results}
We start by establishing a relation between two Pythagorean triples, both of type (ii) according to~\eqref{eq:twoTypes}. 
\begin{lemma}\label{lem:halfAngle}
Let $h \ge 2$ be an integer. Consider the two Pythagorean triples 
\[
[2h-1,2h^2-2h, 2h^2-2h+1] \text{ and } [4h-5, 8h^2-20h+12, 8h^2-20h+13],
\] 
and define 
\[
\phi :=  \arcsec\left ( \frac{2h^2-2h+1}{2h^2-2h} \right)   \text{ and } 
\phi' :=  \arcsec\left ( \frac{8h^2-20h+13}{8h^2-20h+12} \right).
\]
Then, we have $\phi' \ge \phi/2$.
\end{lemma}
\begin{proof}
Since $\phi, \phi' \in (0,\frac\pi2)$ and $h \ge 2$, we have
\begin{equation*}
\begin{split}
\phi' \ge \frac\phi2 
& \iff  \tan (\phi') \ge \tan \left(\frac\phi2 \right) = \frac{1-\cos(\phi)}{\sin(\phi)}\\
& \iff \frac{4h-5}{8h^2-20h+12} \ge \frac{1}{2h-1}\\
& \iff 8h^2 - 14h + 5 \ge 8h^2-20h+12\\
& \iff h \ge \frac76,
\end{split}
\end{equation*}
which proves the claim.
\end{proof}

%
%
%
%

\begin{lemma}\label{thm:heurConstruction}
Let $\nu \in \mathbb{Z}_{++}$ and $\hat \theta_0 :=\pi/2$. 
Consider $\hat \theta_j$, $j=1,\dots,\nu$ as defined in \eqref{eq:heurConstructionAngle}. Then, we have $\hat \theta_j \ge \hat \theta_{j-1}/2 $, for $j=1,\dots, \nu$.
\end{lemma}
\begin{proof}
We will prove this statement considering three cases:

\begin{itemize}
\item $j=1$: In this case, we have $\hat \theta_1 = \arctan(120/119) \ge \pi/{\blue4} = \hat \theta_0 / 2$.

\item $j = 2$: In this case, we have $\hat \theta_2 = \hat \theta_1/2$ since 
\[
\sin (2\hat \theta_2) = 2 \sin(\hat\theta_2) \cos(\hat \theta_2) = 2 \cdot \frac{5}{13} \cdot \frac{12}{13} = \frac{120}{169} = \sin(\hat \theta_1).
\]

\item $j \ge 3$: In this case, we apply Lemma \ref{lem:halfAngle} with $h=2^{j-2}+2$, $\phi = \hat\theta_{j-1}$ and $\phi' =\hat \theta_j $.
\end{itemize}
Hence, the result follows.
\end{proof}

\subsubsection{Proof of Proposition \ref{thm:heurConstruction} and Some Implications}
\label{sec:proofProp3}

Now, we are ready to prove the main result of this section.
\begin{proof}[Proof of Proposition \ref{thm:heurConstruction}]
Due to Lemma \ref{thm:heurConstruction}, $\hat \theta_j$'s satisfy the condition that  $\hat \theta_j \ge \hat \theta_{j-1}/2 $, for $j=1,\dots, \hat \nu_\delta$. Therefore, we have $L^3 \subseteq \tilde P^{\hat \nu_\delta} ( \hat \theta)$. To prove the $(1+\delta)$-approximation guarantee, it suffices to choose the smallest integer $\nu$ such that
\[
\arcsec ( \hat \theta_{\nu} ) \le 1+\delta \iff  \frac{2h^2-2h+1}{2h^2-2h} \le 1+\delta \text{ with } h=2^{\nu-2}+2,
\]
or, equivalently, the smallest $h=2^{\nu-2}+2$ with integer $\nu$ such that
\[
2h^2-2h-\frac1\delta \ge 0 \iff h \ge \frac12 + \sqrt{\frac14+\frac{1}{2\delta}}.
\]
By plugging  $h=2^{\nu-2}+2$ and defining the smallest integer $\nu$ satisfying the above inequalities as~$\hat\nu_\delta$, we obtain \eqref{eq:hatNuDelta}.

Finally, we will prove the last assertion of the proposition. In particular, we have
\begin{equation*}
\begin{split}
\hat \nu_\delta  -  \nu_\delta 
& =  \left\lceil \log_2 \left (-6+2\sqrt{1+2/\delta} \right) \right \rceil - \left\lceil  {\log_2 \frac{\pi}{2\arcsec(1+\delta)}}  \right \rceil \\
& \le \left \lceil \log_2 \frac{ -6+2\sqrt{1+2/\delta} }{ \frac{\pi}{2\arcsec(1+\delta)} } \right \rceil \\
& = \bigg \lceil \log_2 \bigg [ \frac4\pi \underbrace{ ( -3+\sqrt{1+2/\delta } )}_{ \le {{\sqrt{2/\delta}}} } \underbrace{ \arcsec(1+\delta)}_{ \le {2}\sqrt{\delta}} \bigg]  \bigg \rceil \\
& \le \bigg \lceil   \log_2 \frac{8\sqrt{2}}{\pi}  \bigg \rceil \\
 & = 2.
\end{split}
\end{equation*}
Here, the first inequality follows since the function $\lceil \cdot \rceil$ is subadditive,  and the second inequality follows due to the following facts:

\begin{fact}\label{fact:fact1}
$  -3+\sqrt{1+2/\delta } \le   {\sqrt{2/\delta}} $ for $\delta > 0$.
\end{fact}
\begin{proof}
Since $\delta > 0$, we have
\begin{equation*}
\begin{split}
-3+\sqrt{1+2/\delta } \le   {\sqrt{2/\delta}}
& \iff (\sqrt{1+2/\delta })^2 \le  ( {\sqrt{2/\delta}} + 3)^2\\
& \iff  (\delta+2)/\delta \le (2+6\sqrt{2\delta}+9\delta)/\delta \\
& \iff 8\delta + 6\sqrt{2\delta} \ge 0,
\end{split}
\end{equation*}
which proves the claim.
\end{proof}

\begin{fact}\label{fact:fact2}
$  \arcsec(1+\delta)  \le {2}\sqrt{\delta} $ for $\delta > 0$.
\end{fact}
\begin{proof}
Let us define a function $\varphi (\delta) :={2}\sqrt{\delta} -\arcsec(1+\delta)  $ for $\delta \ge 0$.  We will prove the claim by showing that $\varphi(\delta) \ge 0$ for $\delta \ge 0$. 

Firstly, observe that $\varphi$ is a continuous function on $\mathbb{R}_+$ and $\varphi(0)=0$. Secondly,  consider the derivative of $\varphi$,
\[
\varphi'(\delta) = \frac{1}{\sqrt{\delta}} - \frac{1}{(1+\delta) \sqrt{(1+\delta)^2-1}} =  \frac{1}{\sqrt{\delta}} \left (1- \frac{1}{(1+\delta)\sqrt{2+\delta}} \right ),
\]
which is positive for $\delta > 0$, and has the limit
\[
\lim_{\delta \to 0^+} \varphi'(\delta) = +\infty.
\]
Therefore, we conclude that $\varphi$ is a nondecreasing function.
Combining the above facts, we obtain that $\varphi(\delta) \ge 0$ for $\delta \ge 0$. 
\end{proof}
\end{proof}

We remark that the difference between  $\hat\nu_\delta$  and $ \nu_\delta$ in Proposition \ref{thm:heurConstruction}  cannot be improved uniformly as Table \ref{tab:opt vs heur nuDelta} demonstrates that there exist $\delta$ values for which this difference is equal to two.
\begin{table}[H]\small
\caption{Comparison of $\nu_\delta$  and $\hat \nu_\delta$  for different $\delta$ values.}\label{tab:opt vs heur nuDelta}
\centering
\begin{tabular}{c|cc}
\toprule
   $\delta$ &   $\nu_\delta$  & $\hat \nu_\delta$ \\
\midrule



  $ 10^{-4} $&          7 &          9 \\

   $10^{-5} $&          9 &         10 \\

 $ 10^{-6}$ &         11 &         12 \\

   $ 10^{-7}$&         12 &         14 \\


\bottomrule
\end{tabular}  
\end{table}

Finally, we provide the closed-form construction of the  the polyhedral cone $\tilde P^{\nu} (\hat\theta)$  based on \eqref{eq:heurConstruction} and \eqref{eq:heurConstructionAngle} up to $\nu = 14$ in Table  \ref{tab:heurConstruct}. The values of $\hat \nu_\delta$ for $\delta=10^{-1}, 10^{-2}, \dots, 10^{-7}$ are highlighted.

\begin{table}[H]\small
\caption{Closed-form constructions of the polyhedral cone $\tilde P^{\nu} (\hat\theta)$.}\label{tab:heurConstruct}
\centering
\begin{tabular}{c|cc}
\toprule
  $j $ &  $\sec(\hat \theta_j)$ &   $\hat a_j, \hat b_j, \hat c_j$       \\
\midrule
         1 & 1.42016807 & 120,119,169 \\

       {\bf 2} & 1.08333333 &  5,12,13 \\

         3 & 1.04166667 &  7,24,25 \\

         4 & 1.01666667 & 11,60,61 \\

       {\bf  5} & 1.00555556 & 19,180,181 \\

         6 & 1.00163399 & 35,612,613 \\

        {\bf 7} & 1.00044563 & 67,2244,2245 \\

         8 & 1.00011655 & 131,8580,8581 \\

       {\bf  9} & 1.00002982 & 259,33540,33541 \\

      {\bf  10} & 1.00000754 & 515,132612,132613 \\

        11 & 1.00000190 & 1027,527364,527365 \\

       {\bf 12} & 1.00000048 & 2051,2103300,2103301 \\

        13 & 1.00000012 & 4099,8400900,8400901 \\

      {\bf  14} & 1.00000003 & 8195,33579012,33579013 \\
\bottomrule
\end{tabular}  
\end{table}

\subsection{Bounding the Largest Coefficient in $\tilde P^{\hat \nu_\delta} ( \hat \theta) $}
\label{sec:boundCoeff}

We now provide an upper bound for the largest coefficient used in the description of the polyhedral cone $\tilde P^{\hat \nu_\delta} ( \hat \theta) $, which is sublinear in $\delta$.

\begin{proposition}\label{thm:heurBoundCoeff}
Let $\delta \in (0,\frac14)$ and the largest coefficient used in the description of the polyhedral cone $\tilde P^{\hat \nu_\delta} ( \hat \theta) $ be defined as
\[
C_\delta :=  2 \hat h^2 - 2 \hat h + 1 , \text { where } \hat  h := 2^{\hat \nu_\delta - 2} + 2.
\]
Then, we have
\[
C_\delta \le  \frac4\delta + 7 - 6 \sqrt{1+\frac2\delta} \le \frac4\delta.
\]
\end{proposition}
\begin{proof}
Firstly, by plugging in the exact value of  $\hat \nu_\delta $ from \eqref{eq:hatNuDelta}, we obtain  
\[
\hat h = 2^{\hat \nu_\delta - 2} + 2 = 2^{  \left\lceil \log_2 \left (-6+2\sqrt{1+\frac2\delta} \right) \right \rceil -2 }  + 2 \le -1 + \sqrt{1+\frac{2}\delta}.
\]
Then, we have
\[
C_\delta =  2 \hat h^2 - 2 \hat h + 1 =  2 \left(\hat h -\frac12 \right)^2 +\frac12 \le 2 \left( -\frac32 + \sqrt{1+\frac{2}\delta}\right)^2 + \frac12 = \frac4\delta + 7 - 6 \sqrt{1+\frac2\delta}.
\]
Finally, since $\delta < \frac14 $, we deduce that $7 - 6 \sqrt{1+\frac2\delta} \le 0$.  Hence, we prove that $C_\delta \le 4/\delta$, which concludes the proof.
\end{proof}
Notice that  the assertion of Proposition \ref{thm:heurBoundCoeff} and the computed values  in Table  \ref{tab:heurConstruct}  are in full accordance.

\subsection{Closed-Form Rational Outer-Approximation of  $L^N$ and Bounding Its Largest Coefficient}
\label{sec:extendLn'}

%

Now, we are ready to give a closed-form rational outer-approximation of $L^N$. Given $\epsilon >0$, let~$\delta_\epsilon$  denote the largest value of $\delta$  satisfying~\eqref{eq:deltaEpsilon} and $\hat \nu_{\delta_\epsilon}$ be selected according to~\eqref{eq:hatNuDelta}. Consider $\hat \theta_j$, $j=1,\dots,\hat \nu_{\delta_\epsilon}$ as defined in \eqref{eq:heurConstructionAngle}.
Let us define the polyhedral cone
\[
\bar Q^N (\hat \theta) := \{x\in\mathbb{R}^{n} : \exists y_{k,i} : \eqref{eq:towerLinear} , (y_{k,2i-1}, y_{k,2i}, y_{k+1,i}) \in \tilde P^{\hat \nu_{\delta_\epsilon}} ( \hat \theta)   \ k=0,\dots,K-1, i=1,\dots, 2^{K-k-1} \},
\]
with $K = \lceil \log_2 (N-1) \rceil$. Then, we clearly have
\[
L^N \subseteq \bar Q^N(\hat \theta)  \subseteq (1+\epsilon) L^N.
\]
Notice that $\bar Q^N(\hat \theta) $ would require $(2\hat\nu_{\delta_\epsilon}+3)(N-2)$ many additional variables and  $(3\hat \nu_{\delta_\epsilon}+6)(N-2)$ many constraints in its description. 
Let us now compare the size of $ Q^N( \theta) $ as defined in Section \ref{sec:extendLn}  and $\bar Q^N(\hat \theta) $. Due to the last assertion of Proposition \ref{thm:heurConstruction},  that is $\hat \nu_\delta  \le  \nu_\delta + 2$, we conclude that the latter require at most $4(N-2)$ more variables and $6(N-2)$ more constraints. Hence, the increase in the size of the closed-form outer-approximation is linear in the dimension of the cone.

We conclude our paper by providing an upper bound on the largest coefficient in the description of $\bar Q^N (\hat \theta) $, which is sublinear in $\epsilon$ and logarithmic in $N$.
\begin{proposition}\label{thm:heurBoundCoeffLn}
Let $\epsilon\in(0,1)$. Consider $L^N$ and its outer-approximation $\bar Q^N (\hat \theta) $ as defined above.
Then,  the largest coefficient used in the description of the polyhedral cone $ \bar Q^N ( \hat \theta )$ is upper bounded by
\[
 \frac{4 \lceil \log_2 (N-1) \rceil }{ (\ln 2) \epsilon}.
\]
\end{proposition}
\begin{proof}
Let $\delta_\epsilon$  denote the largest value of $\delta$  satisfying the relation in \eqref{eq:deltaEpsilon}. Then, we have
\[
\delta_\epsilon = (1+\epsilon)^{1/ \lceil \log_2 (N-1) \rceil } - 1 \ge ( 2^{ \lceil \log_2 (N-1) \rceil} - 1) \epsilon \ge \frac{\ln 2}{ \lceil \log_2 (N-1) \rceil} \epsilon.
\]
Here, the first inequality follows since the function $(1+\epsilon)^{1/K } - 1$ is concave in $\epsilon$,  and hence, it can be lower bounded by a linear function passing through the points at $\epsilon=0$ and $\epsilon=1$. The second inequality follows due to the following fact:

\begin{fact}
$  2^{1/K} - 1 \ge (\ln2) / K$ for $K >0$.
\end{fact}
\begin{proof}
We first note that the function $2^\tau - 1$ is convex on  $ \mathbb{R}_+$, therefore, it can be lower bounded by the gradient inequality at the point $\tau=0$. Hence, we have $2^\tau - 1 \ge  (\ln2) \tau$. Then, the statement follows through the variable transformation $\tau := 1/ K$ for $K > 0$.
\end{proof}

To conclude the proof, we use Proposition \ref{thm:heurBoundCoeff} with $\delta_\epsilon$ and obtain that the largest coefficient in the description of $\bar Q^N (\hat \theta) $ is upper bounded by
\[
 \frac4{\delta_\epsilon} \le  \frac{4 \lceil \log_2 (N-1) \rceil }{ (\ln 2) \epsilon},
\]
which is the upper bound on the largest coefficient in the description of $\tilde P^{\hat \nu_{\delta_\epsilon}} ( \hat \theta)$.
\end{proof}

\section{Third Outer-Approximation: Achieving the Best Accuracy Given an Upper Bound on Coefficients}
\label{sec:deltaGivenM}

In the previous sections, we analyze the rational polyhedral outer-approximations of $L^N$ when the accuracy $\epsilon$ is specified. Although these outer-approximations have provably small sizes in terms of the number of variables and linear inequalities, the integers used in their inequality description grow large. In this section, we will analyze the \textit{reverse} problem: Given an upper bound on the largest coefficient, obtain the smallest size formulation among the approximations which provide the best accuracy, that is, the one with the smallest $\epsilon$. We solve this problem in three steps: After establishing some preliminary results in Section~\ref{sec:deltaGivenM_prelim}, we obtain the best possible approximation accuracy as a function of the upper bound on the coefficients and provide an approximation which attains this accuracy in Section~\ref{sec:deltaGivenM_closed}. Finally, we propose another scheme which can potentially reduce the size of the inequality description of the outer-approximating polyhedron while retaining the same level of accuracy in Section~\ref{sec:deltaGivenM_optimized}.
 Throughout this section, we will primarily focus on the analysis of $L^3$ and mention the implications for general $L^N$ at the end.

\subsection{Preliminary Results}
\label{sec:deltaGivenM_prelim}
Recall that our aim in this section is to compute the best approximation accuracy of the rational polyhedral outer-approximation of $L^3$ given an upper bound on the coefficients of the inequality description as $C \in \mathbb{Z}_+$.  Throughout this section, we will assume that $C \ge 169$ since smaller values of $C$ does not produce a reasonable approximation guarantee.
We start our analysis with the following three lemmas:
\begin{lemma}\label{lem:optSolGivenC1}
Consider the following optimization problem: 
\begin{equation}\label{eq:optSolGivenC1}
\psi_1(C) := \min_{ m\in\mathbb{Z},  n\in\mathbb{Z} } \bigg\{ \frac{m^2+n^2}{2mn} : \ m^2+n^2 \le C, \ m-1 \ge n \ge 1 \bigg\}.
\end{equation}
Then, the optimal solution $(m_1^*, n_1^*)$  satisfies the condition $ n_1^* =  m_1^* - 1$ with $m_1^* = \lfloor \frac12 ( 1 +  \sqrt{ 2C - 1} \rfloor$.
\end{lemma}
\begin{proof}
Observe that the objective function of problem~\eqref{eq:optSolGivenC1} can be written as $\frac12 (\frac mn + \frac nm)$. Noting that the function $\omega_1(\zeta) = \zeta + 1/\zeta $ is increasing for $\zeta \ge 1$, we deduce that the ratio $m / n $ should be chosen as small as possible in an optimal solution. Considering the constraint $m \ge n+1$, we conclude that the condition $ n_1^* =  m_1^* - 1$ must be satisfied. Finally, the exact value of $m_1^* $ can be computed as the largest value of $m$ satisfying the constraint $m^2 + (m-1)^2 \le C$.
\end{proof}

\begin{lemma}\label{lem:optSolGivenC2}
Consider the following optimization problem: 
\begin{equation}\label{eq:optSolGivenC2}
\psi_2(C) := \min_{ m\in\mathbb{Z},  n\in\mathbb{Z} } \bigg\{ \frac{m^2+n^2}{m^2 - n^2} : \ m^2+n^2 \le C, \ m-1 \ge n \ge 1 \bigg\}.
\end{equation}
Then, the optimal solution is  $(m_2^*, n_2^*) = ( \lfloor   \sqrt{ C - 1} \rfloor , 1 )$.
\end{lemma}
\begin{proof}
Observe that the objective function of problem~\eqref{eq:optSolGivenC2} can be written as $1 + 2 / ( m^2/n^2  - 1 ) $. Noting that the function $\omega_2(\zeta) = 1 + 2/ (\zeta^2-1) $ is decreasing for $\zeta > 1$, we deduce that the ratio $m / n $ should be chosen as large as possible in an optimal solution. Considering the constraint $n\ge1$, we conclude that the condition $ n_2^* = 1$ must be satisfied. Finally, the exact value of $m_2^* $ can be computed as the largest value of $m$ satisfying the constraint $m^2 + 1^2 \le C$.
\end{proof}


\begin{lemma}\label{lem:optSolGivenCtriple}
Consider the following optimization problem: 
\begin{equation}\label{eq:optSolGivenC}
\psi(C) := \min_{ a\in\mathbb{Z},  b\in\mathbb{Z},  c\in\mathbb{Z}  } \bigg\{ \frac{c}{b} : \ a^2+b^2 = c^2, c \le C, \ a,b,c \ge 1 \bigg\}.
\end{equation}
Then, we have $\psi(C) = \psi_1(C)$, 
where $\psi_1(\cdot)$ is defined as in~\eqref{eq:optSolGivenC1}.
\end{lemma}
\begin{proof}
Since any  integer Pythagorean triple $[a, \ b, \ c]$ can be written in the form  $ [ m^2-n^2, \ 2mn, \ m^2+n^2] $ for some positive integers $m > n$, we first note that  $ \psi(C) = \min\{ \psi_1(C) , \psi_2(C)  \}$ due to  Lemmas~\ref{lem:optSolGivenC1} and \ref{lem:optSolGivenC2}.

Our proof strategy  to prove   $ \psi(C) = \min\{ \psi_1(C) , \psi_2(C)  \} = \psi_1(C) $ is to show that that an upper bound for  $\psi_1(C) -1 $ is smaller than a lower bound for $ \psi_2(C) -1 $. 
Firstly, consider $m_1^*$ as in the proof of Lemma~\ref{lem:optSolGivenC1}, and let $m_1'=  \frac12 ( - 1 +  \sqrt{ 2C - 1}) \le m_1^* $. Then, we  have 
\begin{equation*}
\begin{split}
{\psi_1(C)-1} 
 =    \frac{ {(m_1^*)}^2+ {(m_1^*-1)}^2}{2m_1^* (m_1^*-1)}  - 1
 = \frac{1}{ 2 m_1^* (m_1^*-1) }     
  \le \frac{1}{ 2 m'_1 (m'_1-1) }     
 = \frac{1}{ 1+C-2\sqrt{2C-1} } .
\end{split}
\end{equation*}
Secondly,  consider $m_2^*$ as in the proof of Lemma~\ref{lem:optSolGivenC2}, and let $m_2'=  \sqrt{ C - 1} \ge m_2^* $. Then, we  have 
\begin{equation*}
\begin{split}
{\psi_2(C)-1} 
 =    \frac{ {(m_2^*)}^2+ 1}{ {(m_2^*)}^2 - 1}  - 1
 =    \frac{ 2}{ {(m_2^*)}^2 - 1} 
 \ge   \frac{ 2}{ {(m'_2)}^2 - 1} 
 = \frac{2}{ C - 2} .
\end{split}
\end{equation*}
Finally, it is easy to show that $ \psi_1(C)-1 \le \frac{1}{ 1+C-2\sqrt{2C-1} } \le   \frac{2}{ C - 2} \le \psi_2(C)-1$ for $C \ge 23$, which satisfies our standing assumption in this section that $C\ge 169$.
 Hence, the result follows.
\end{proof}
We will use Lemma~\ref{lem:optSolGivenCtriple}  to compute the best  accuracy of the outer-approximation given the upper bound $C$, which is sublinear in $C$ as the proof indicates, in the next subsection.

\subsection{A Closed-Form Construction with the Best Approximation Accuracy}
\label{sec:deltaGivenM_closed}

In this section, we will construct a rational polyhedral outer-approximation of $L^3$ such that the size of the coefficients in its description is upper bounded  b y$C$ with an approximation accuracy of $\delta=\psi(C)-1$. Our construction in this part is the \textit{reverse} of the closed-form construction defined in Section~\ref{sec:heuristicConstruction}. In particular, let us define
\begin{equation}\label{eq:nu given C}
\check \nu^C = 2 + \lfloor\log_2 (\check h_{\check \nu^C}-2) \rfloor
\end{equation}
where
\[
\check h_{\check \nu^C} =  \bigg\lfloor \frac12 ( 1 +  \sqrt{ 2C - 1} ) \bigg\rfloor .
\]
Let us consider the recursion
\begin{equation}\label{eq:reverse recursion}
\check  h_{j-1} := \bigg\lceil \frac{\check h_j+1}{2} \bigg\rceil ,
\end{equation}
for $j = \nu^C, \dots, 3 $.
Consider the Pythagorean triples defined as
\begin{equation}\label{eq:heurConstructionGivenC}
(\check a_j, \check b_j, \check c_j) = \begin{cases}
(120, \ 119, \ 169) & \text{ if } j=1 \\
(2\check h_j-1,2\check h_j^2-2\check h_j, 2\check h_j^2-2\check h_j+1 )   & \text{ if } j = 2, \dots, \check\nu^C
\end{cases}
\end{equation}
and the corresponding angles
\begin{equation}\label{eq:heurConstructionAngleGivenC}
\check \theta_j := \arcsec   ( {\check c_j}/{ \check b_j }  ), \ j=1,\dots,\check\nu^C.
\end{equation}
Notice that $(\check a_j, \check b_j, \check c_j) $ are of the  type (ii) according to~\eqref{eq:twoTypes}  for $j=2,\dots,\check\nu^C$. 
Before proving the correctness of this construction in Proposition~\ref{thm:heurConstructionGivenC}, we provide 
some sample results  in Table~\ref{tab:10to5-10to7Heur} for $C=10^{5}, 10^{6}, 10^{7} $. 


\begin{table}[H]\footnotesize
\centering
\caption{Constructions of the polyhedral cone $ \tilde P^{\check \nu^C} ( \check \theta)$ for different $C$ values.}\label{tab:10to5-10to7Heur}
\begin{tabular}{c|cc|cc|cc}
\toprule

           & \multicolumn{ 2}{c|}{ $C=10^5$ ($\check\nu^C =9$)} & \multicolumn{ 2}{c|}{ $C=10^6$  ($\check\nu^C =11$) } & \multicolumn{ 2}{c}{ $C=10^7$  ($\check\nu^C =13$)} \\

      $j $ &  $\sec(\check \theta_j)$ &   $\check a_j,\check b_j,\check c_j$ & $\sec(\check \theta_j)$ &   $\check a_j,\check b_j,\check c_j$ & $\sec(\check \theta_j)$ &   $\check a_j,\check b_j,\check c_j$ \\

\midrule
         1 &    1.42017 & 120,119,169 &   1.420168 & 120,119,169 &  1.4201681 & 120,119,169 \\

         2 &    1.08333 &    5,12,13 &   1.083333 &    5,12,13 &  1.0833333 &    5,12,13 \\

         3 &    1.02500 &    9,40,41 &   1.041667 &    7,24,25 &  1.0416667 &    7,24,25 \\

         4 &    1.00893 & 15,112,113 &   1.011905 &   13,84,85 &  1.0166667 &   11,60,61 \\

         5 &    1.00238 & 29,420,421 &   1.003205 & 25,312,313 &  1.0055556 & 19,180,181 \\

         6 &    1.00062 & 57,1624,1625 &   1.000906 & 47,1104,1105 &  1.0014620 & 37,684,685 \\

         7 &    1.00016 & 113,6384,6385 &   1.000242 & 91,4140,4141 &  1.0003968 & 71,2520,2521 \\

         8 &    1.00004 & 225,25312,25313 &   1.000062 & 179,16020,16021 &  1.0001006 & 141,9940,9941 \\

         9 &    1.00001 & 447,99904,99905 &   1.000016 & 355,63012,63013 &  1.0000253 & 281,39480,39481 \\

        10 &            &            &   1.000004 &  707,249924,249925          &  1.0000064 & 561,157360,157361 \\

        11 &            &            &   1.000001 &  1413,998284,998285          &  1.0000016 & 1119,626080,626081 \\

        12 &            &            &            &            &  1.0000004 & 2237,2502084,2502085 \\

        13 &            &            &            &            &  1.0000001 & 4471,9994920,9994921 \\

\bottomrule
\end{tabular}  
\end{table}

Our main result in this section is the following proposition:

\begin{proposition}\label{thm:heurConstructionGivenC}
%
%
Let $\check \theta_0 =\pi/2$.
Consider the polyhedral set $\tilde P^{\check  \nu^C} ( \check \theta)$ with $\check\nu^C$ as defined in~\eqref{eq:nu given C}
and $\check \theta_j$, $j=1,\dots, \check \nu_C$ as defined in \eqref{eq:heurConstructionAngleGivenC}. 
Then, we have
\[
L^3 \subseteq \tilde P^{ \check\nu^C} ( \check \theta) \subseteq  \psi(C) L^3.
\]
Moreover,
\[
 \check\nu^C   -  \nu_{\psi(C)-1}  \le 1.
\]
\end{proposition}
\begin{proof}
Firstly, one can check that the angles $\check \theta_j $ satisfy the relation $\check \theta_j \ge \check \theta_{j-1}/2 $, for $j=1,\dots, \check\nu^C$ (we omit the proof due to its similarity with  the proofs of Lemmas~\ref{lem:halfAngle} and \ref{thm:heurConstruction}). Therefore, we conclude that $L^3 \subseteq \tilde P^{ \check\nu^C} ( \check \theta)$. Also, due to the choice of $(\check b_{\check \nu^C}, \check c_{ \check\nu^C})$ and the relation 
\[
\psi(C) = \frac{\check c_{ \check\nu^C} }{ \check b_{\check \nu^C} },
\]
established in Lemma~\ref{lem:optSolGivenCtriple}, we have $\tilde P^{ \check\nu^C} ( \check \theta) \subseteq  \psi(C) L^3$.

Finally, we will prove the last assertion of the proposition. Let $m = \check h_{\check\nu^C} =   \lfloor \frac12 ( 1 +  \sqrt{ 2C - 1} \rfloor$ as in Lemma~\ref{lem:optSolGivenC1}.  Then, using equations~\eqref{eq:nuOptimal} and \eqref{eq:nu given C}, we have
\begin{equation*}
\begin{split}
\check \nu^C   -  \nu_{\psi(C)-1} 
& =   2 + \lfloor\log_2 (m-2) \rfloor -  \left\lceil  {\log_2 \frac{\pi}{2\arcsec( \psi(C))}}  \right \rceil \\
& \le   2 + \log_2 (m-2)   -     {\log_2 \frac{\pi}{2\arcsec( \frac{1+2m^2-2m}{2m^2-2m} )}}    \\
& = 2 +  \log_2 \bigg [ \frac2\pi  (m-2) \underbrace{ \arcsec(1+\frac{1}{2m^2-2m})}_{ \le {2}/\sqrt{2m^2-2m}} \bigg]    \\
& \le 2 +  \log_2 \bigg [ \frac4\pi  \frac{(m-2)}{\sqrt{2m^2-2m}} \bigg]    \\
& \le 2 +  \log_2 \bigg ( \frac{4}{\pi \sqrt{2}}    \bigg )    \\
& \le 1.
\end{split}
\end{equation*}
Here, the first inequality follows by the basic properties of rounding operators, the second inequality follows due to Fact~\ref{fact:fact2} and the third inequality follows  due to the following fact:
\begin{fact}\label{fact:fact4}
$  \frac{m-2}{ \sqrt{2m^2 - 2m} } \le  \frac{1}{\sqrt{2}} $ for $m \ge 2$.
\end{fact}
\begin{proof}
Since $m \ge 2$, we have
\begin{equation*}
\begin{split}
 \frac{m-2}{ \sqrt{2m^2 - 2m} } \le  \frac{1}{\sqrt{2}}
& \iff   {\sqrt{2}} (m-2) \le { \sqrt{2m^2 - 2m} }   \\
& \iff   2 (m^2-4m+4) \le { {2m^2 - 2m} }   \\
& \iff  \frac43 \le m,
\end{split}
\end{equation*}
which proves the claim.
\end{proof}

\end{proof}

%

We have two remarks related to  Proposition~\ref{thm:heurConstructionGivenC}: Firstly, the approximation accuracy $\delta = {\blue\psi}(C) - 1$ given in Proposition~\ref{thm:heurConstructionGivenC} cannot be improved due to  Lemma~\ref{lem:optSolGivenCtriple}. The reason is that  $\psi(C)$  is the smallest  value of the  the arcsecant of an angle of a right triangle corresponding to a Pythagorean triple whose hypotenuse is upper bounded by $C$. 
Secondly, the size of the construction in terms of $\check\nu^C$ is larger than the size of the optimal construction under the same accuracy $\delta = {\blue\psi}(C) - 1$ from Section~\ref{sec:extendL3} by only at most one in the worst case. Note, however, that the value of  $\nu_{\psi(C)-1}$ may not be achievable since it might possibly require coefficients larger than $C$. 

In Table~\ref{tab:opt vs heur nuC}, we report the values of $\nu_{\psi(C)-1}$  and $ \check \nu^C$  for different $C$ values.
As predicted by Proposition~\ref{thm:heurConstructionGivenC}, the difference between these two quantities is at most one. In fact,  $\nu_{\psi(C)-1} > \check\nu^C$ for only $C=10^7$ among the upper bounds considered here. This has motivated us to further explore different ways to construct a different outer-approximation as explained in the next subsection.

\begin{table}[H]\small
\caption{Comparison of $\nu_{\psi(C)-1}$  and $ \check \nu^C$  for different $C$ values.}\label{tab:opt vs heur nuC}
\centering
\begin{tabular}{c|cc}
\toprule
   $C$ &   $\nu_{\psi(C)-1}$  & $ \check\nu^C$ \\
\midrule

  $ 10^{5} $&          9 &          9 \\

   $10^{6} $&          11 &         11 \\

 $ 10^{7}$ &         12 &         13 \\

\bottomrule
\end{tabular}  
\end{table}

\subsection{An Improved Construction with A Smaller Size}
\label{sec:deltaGivenM_optimized}

Although the recursion given in~\eqref{eq:reverse recursion} leads to an outer-approximation with the desired accuracy $\psi(C)-1$ given an upper bound $C$ by guaranteeing the condition  $\check \theta_j \ge \check \theta_{j-1}/2 $, for $j=1,\dots, \check\nu^C$, it only utilizes Pythagorean triples of  type (ii)  according to~\eqref{eq:twoTypes}. However, it might be possible to increase the angles $ \theta_{j-1} $ more aggressively (recall that we are deciding them in a backward manner) while keeping the condition  $  \theta_{j-1} \le  2 \theta_{j} $ intact so that the size of the eventual formulation is smaller in terms of $\nu^C$. According to Proposition~\ref{thm:heurConstructionGivenC}, this might be possible as well.

For this purpose, we propose  Algorithm~\ref{alg:imprConstGivenC}, which tries to find the next angle $\theta_{j-1}$, which is as close as  possible to the \textit{target} angle $2 \theta_{j}$ at each step. This entails solving the nonconvex nonlinear integer program~\eqref{eq:optSolGivenC_alg}. We note that this problem can be reformulated by moving the objective function to the constraints by defining a new variable and solving the resulting nonconvex quadratically constrained integer program by the solver Gurobi (version 9).
\begin{algorithm}
\caption{Improved construction of $\tilde P^{ \nu^C} (\theta)$.}
\label{alg:imprConstGivenC}
\begin{algorithmic}
\STATE Obtain $(a_0',b_0',c_0')\in\mathbb{Z}^3$ as an optimal solution of problem~\eqref{eq:optSolGivenC} and set $\theta'_0= \arcsec(c'_{0}/b'_{0})$.
\STATE Set $j=0$ and $ \texttt{target}  = 2 $.
\WHILE{$ \texttt{target}  > 0 $}
\STATE Set $\bar C = \max\{169, c'_{-j}\}$.
\STATE Set $\texttt{target} = \sec(2\theta'_{-j})$.
\STATE Set $j = j+1$.
\STATE Obtain
\begin{equation}\label{eq:optSolGivenC_alg}
(a_{-j}',b_{-j}',c_{-j}') = \argmax_{ a\in\mathbb{Z},  b\in\mathbb{Z},  c\in\mathbb{Z}  } \bigg\{ \frac{c}{b} : \ a^2+b^2 = c^2, c \le \bar C, \ \frac{c}{b} \le \texttt{target} , \ a,b,c \ge 1 \bigg\},
\end{equation}
and set  $\theta'_{-j}= \arcsec(c'_{-j}/b'_{-j})$.
\ENDWHILE
\STATE Set $\nu^C = j + 1$.
\FOR{$j'=1,\dots,\nu^C$}
\STATE Set  $\theta_{j'} = \theta'_{j' -\nu^C}$.
\ENDFOR

\end{algorithmic}
\end{algorithm}

We run Algorithm~\ref{alg:imprConstGivenC} with $C=10^7$ since the size of the formulations obtained from previous subsection is optimal  for the other values given in Table~\ref{tab:opt vs heur nuC}. We observe that there exists an outer-approximation with the same approximation accuracy with a smaller (and optimal) size of $\nu^C=12$ for $C=10^7$. We give the coefficients of such a construction in Table~\ref{tab:10to7Impr}.

\begin{table}[H]\footnotesize
\centering
\caption{An improved construction of the polyhedral cone $\tilde P^{ \nu^C} (\theta)$  for $C=10^7$ (with $\nu^C=12$).}\label{tab:10to7Impr}
\begin{tabular}{c|cc}
\toprule
         $j $ &  $\sec(  \theta_j)$ &   $  a_j,  b_j,  c_j$  \\
\midrule
         1 &  1.5568182 & 105,88,137 \\

         2 &  1.1038961 &   36,77,85 \\

         3 &  1.0250000 &    9,40,41 \\

         4 &  1.0061920 & 36,323,325 \\

         5 &  1.0016340 & 35,612,613 \\

         6 &  1.0004082 & 140,4899,4901 \\

         7 &  1.0001020 & 280,19599,19601 \\

         8 &  1.0000255 & 560,78399,78401 \\

         9 &  1.0000064 & 559,156240,156241 \\

        10 &  1.0000016 & 2236,1249923,1249925 \\

        11 &  1.0000004 & 4472,4999695,4999697 \\

        12 &  1.0000001 & 4471,9994920,9994921 \\

\bottomrule
\end{tabular}  
\end{table}



We conclude this section by the following remark:
We note that given an upper bound on the coefficients of the polyhedral representation of the outer-approximation as $C$, the best approximation guarantee we can obtain for $L^N$ is
\[\epsilon =  \psi(C)^{\lceil \log_2(N-1) \rceil} - 1 ,\]
where $\psi(C)$ is defined as in~\eqref{eq:optSolGivenC}. It is clear that if $C$ is kept constant, the accuracy of the approximation will diminish with larger values of $N$.

\section{Applications}
\label{sec:applications}

In this section, we  discuss two possible applications of rational polyhedral outer-approximations of the second-order cone { with a more theoretical focus, and another application with a computational focus. All} of these applications  involve the  set,
\begin{equation}\label{eq:setS}
\mathcal{S} := \bigcap_{i=1}^I \mathcal{B}_{R_i} (u_i),
\end{equation}
with $u_i \in  \mathbb{Z}^N$ and $R_i \in \mathbb{Z}_{++} $ for $i=1,\dots, I$, where the set  $\mathcal{B}_R (u) := \{ x\in\mathbb{R}^N : \|x-u\|_2 \le R\}$ denotes the  ball centered at $u \in \mathbb{R}^N$ with radius $R > 0 $. We will also assume that $\cS \cap \mathbb{Z}^N \neq \emptyset$. We choose the set $\cS$ for  ease of illustration although our derivations in this section remain valid for ellipsoids defined by integral data.

In the first application discussed in Section~\ref{sec:intgap}, we analyze the problem of estimating the integrality gap of optimizing a linear function over the set $\cS \cap \mathbb{Z}^N $.  In the second application discussed in Section~\ref{sec:count}, we analyze the problem of counting the number of integral vectors in the set $\cS$. Motivated by the fact that both problems are well-understood for rational polyhedra, our approach in this section is to  transform these problems  into their approximate or equivalent versions involving the rational polyhedral outer-approximations of the set $\cS$. This will  allow us to apply the well-known results from the literature for rational polyhedra. We will also discuss the importance of the rationality assumption for both applications. { Finally, we compare the efficiency of minimizing a linear function over the sets  $\cS  $ and  $\cS \cap \mathbb{Z}^N $ via different polyhedral outer-approximations   in Section~\ref{sec:optimization}. Our preliminary experiments suggest that there might be potential advantages of using the outer-approximations defined with rational or integer coefficients compared to the one that involves irrational coefficients.}

%
%
%

{Our analysis in this section relies on the following result:
\begin{proposition}\label{prop:SvsT}
Consider the set $\cS \subseteq \mathbb{R}^N$ as defined above. Then, there exists a set $\cT \subseteq \mathbb{R}^N$ such that 
\begin{enumerate}
\item $\cT$ is a  polytope defined by integral data,
\item $\cT \supseteq \cS$,
\item $  \cT \cap \mathbb{Z}^N =  \cS \cap \mathbb{Z}^N$.
\end{enumerate} 
\end{proposition}
\begin{proof}
Firstly, let us choose $\epsilon_i \in \mathbb{Q}$ such that $\epsilon_i \in (0, \sqrt{1+1/R_i^2} - 1)$ for $i=1,\dots,I$. Then, due to Proposition~\ref{thm:outerL3rational} and the tower-of-variables construction from Section~\ref{sec:extendLn}, we can construct a polyhedral rational outer-approximation $\cQ_i$ of the ball $\mathcal{B}_{R_i} (u_i)$ satisfying
\[
\mathcal{B}_{R_i} (u_i) \subseteq \cQ_i  \subseteq (1+\epsilon_i) \mathcal{B}_{R_i} (u_i) .
\]
Moreover, due to our choice of $\epsilon_i$, we have $\mathcal{B}_{R_i} (u_i) \cap \mathbb{Z}^N = \mathcal{Q}_i \cap \mathbb{Z}^N$ for $i=1,\dots,I$ since
\[
\{ x\in\mathbb{Z}^N: \ \| x - u_i \|_2 \le R_i  \} = \{ x\in\mathbb{Z}^n: \ \| x - u_i \|_2 < \sqrt{R_i^2 + 1}  \}.
\]
Finally, we define the set
\[
\cT := \bigcap_{i=1}^I \cQ_i,
\]
which satisfies the three desired properties stated in the proposition.
\end{proof}

\subsection{Estimating the Integrality Gap of Optimizing a Linear Function over $\cS \cap \mathbb{Z}^N$}
\label{sec:intgap}

The integrality gap of a mixed-integer program, which is defined as the difference between  its optimal value and the optimal value of its continuous relaxation, is an important measure of both theoretical and practical interest. More precisely, let $\mathcal{X} \subseteq \mathbb{R}^M$ and 
$\alpha \in \mathbb{R}^M$. Then, we define the integrality gap of minimizing a linear function $\alpha^T x$  over $\cX\cap\cL \neq \emptyset$, where $\cL =  \mathbb{Z}^N \times  \mathbb{R}^{M-N}$, as
\[
IG_{\cL} (\mathcal{X} ) = \inf_{x \in \cX \cap  \cL  }   \alpha^T x   - \inf_{x \in \cX  } \alpha^T x . 
\]
The  integrality gap of a mixed-integer linear program defined by integral data is well-studied. 
For instance, if $\cX = \{x \in \mathbb{R}^M: Ex \ge f \} $ where  $E$ is integral, a classical result from \cite{cook1986sensitivity} states that 
\[
IG_{\cL} (\mathcal{X} )  \le M \Delta(E) \|\alpha \|_1,
\]
where $\Delta(E)$ is the maximum of the absolute values of the subdeterminants of the integral matrix $E$. In a recent paper \cite{paat2020distances}, it has been shown that the dependence on the dimension can be reduced to the number of integer variables $N$ instead of the number of total variables $M$.

We will now analyze the integrality gap of minimizing a linear function $\alpha^T x$  over the set  $\cS \cap \mathbb{Z}^N$, that is, $ IG_{\mathbb{Z}^N} ( \cS) $,  
where $\alpha \in \mathbb{R}^N$.
To the best of our knowledge, the integrality gap of an integer program defined over the integral vectors in the intersection of balls has not been analyzed in detail.  
For this purpose, we will utilize Proposition~\ref{prop:SvsT}  and conclude  that  $ IG_{\mathbb{Z}^N} ( \cS)  \le   IG_{\mathbb{Z}^N} ( \cT) $. Notice that $\cT$ is a polytope defined by integral data but its description contains additional continuous variables ($\xi^j$ and $\eta^j$). Therefore, we can use the result from \cite{paat2020distances} and obtain an upper bound for $ IG_{\mathbb{Z}^N} ( \cT) $, and consequently for $ IG_{\mathbb{Z}^N} ( \cS) $.

We remark that the proofs from  \cite{cook1986sensitivity, paat2020distances} explicitly use the fact that the underlying polyhedron is defined by integral data. Therefore, we would not be able to obtain a similar result for estimating the integrality gap over $  \cS \cap \mathbb{Z}^N$
 if Ben-Tal and Nemirovski's outer-approximation is used for $\cT$ since it contains irrational coefficients. 


%
%
%

\subsection{Counting the Number of Integral Vectors in $\cS$}
\label{sec:count}

The problem of counting the number of integral vectors in a convex set goes back to at least 	Gauss, who made the first progress to estimate the number of integral vectors in a circle of integer radius (this problem is now called the Gauss circle problem). 
 In 1994, Barvinok  showed, in his breakthrough paper, that the number of integral vectors in a rational polytope can be computed exactly in polynomial time when the dimension is fixed \cite{barvinok1994} (previously, similar results were known for polytopes of dimension at most four \cite{dyer1991}). Since then, practical implementations of his algorithm have been developed (see, for example, \cite{latte2004, verdoolaegebarvinok}).

Now suppose that we are interested in counting the number of integral vectors in the set $\cS$. At this point, we will again utilize Proposition~\ref{prop:SvsT}
since $   \cS \cap \mathbb{Z}^N =  \cT \cap \mathbb{Z}^N $, and count the number of integral vectors in the set $\cT$ instead. Although it is tempting to directly use Barvinok's algorithm for this task, there are two subtleties. Firstly, the inequality description of the polyhedral set $\cT$ is given in an extended space, containing  additional continuous variables ($\xi^j$ and $\eta^j$). However, this issue can be resolved by a linear transformation of variables
\[
\xi^j \leftarrow \frac{1}{ \prod_{j'=1}^j c_{j'} } \xi^j \text{ and } \eta^j  \leftarrow \frac{1}{ \prod_{j'=1}^j c_{j'} } \eta^j,
\]
and requiring $ \xi^j$ and $ \eta^j$ variables to be integer (note that this restriction does not affect the proof of 
Proposition~\ref{thm:outerL3general}, hence, the outer-approximation remains valid). Secondly, although all the variables including the additional variables are integer in the new extended formulation for $\cT$, our aim was to count the integral vectors in the original space, which requires a projection. However, this issue has already been resolved  by Barvinok and Woods  \cite{barvinok2003proj}, and their integer projection algorithm has been implemented as well \cite{koppe2008implementation}.

We remark that the rationality of the underlying polyhedron is quite important for  Barvinok-type results, which do not extend fully to irrational polyhedra  from both theoretical and implementation perspectives \cite{barvinok2007lecture}.

}

{
\subsection{Computational Comparison of Different Outer-Approximations}
\label{sec:optimization}

In this section, we present some preliminary computational experiments, which suggest  that there may be potential advantages of the outer-approximations proposed in this paper. Let us now explain the experimental setup in detail and discuss our findings.

Given $\beta\in\mathbb{R}^N$, suppose that our aim  is to solve
\[
 \min_{ x \in \cS }   \beta^T x  \quad \text{ and } \quad  \min_{ x \in \cS \cap \mathbb{Z}^N }   \beta^T x,
\]
by outer-approximating the set $\cS$ defined in \eqref{eq:setS} with a polyhedron $\mathcal{R}$. Hence, we are interested in solving
\[
 \min_{ x \in \mathcal{R}}   \beta^T x  \quad \text{ and }  \quad \min_{ x \in \mathcal{R} \cap \mathbb{Z}^N  }   \beta^T x,
\]
which we will refer  as \textit{LP Version} and \textit{IP Version}, respectively. 
We compare the following  three alternative formulations for  $\mathcal{R}$:
\begin{itemize}
\item
$\mathbb{R}$-Formulation \cite{ben2001polyhedral}: Formulation~\eqref{eq:outer} with real coefficients specified in~\eqref{eq:nuOptimal}.
\item
$\mathbb{Q}$-Formulation: Formulation~\eqref{eq:outer} with rational coefficients specified in Table~\ref{eq:optConstruct-4-6}.
\item
$\mathbb{Z}$-Formulation: Formulation~\eqref{eq:outerInteger} with integer coefficients specified in Table~\ref{eq:optConstruct-4-6}.
\end{itemize}
We remark that $\mathbb{R}$-Formulation and $\mathbb{Q}$-Formulation are subject to rounding errors due to the finite precision of computers since we cannot  input exact real and rational numbers to a solver such as Gurobi.

In our computational experiments, we vary the dimension $N$, and choose the number of balls as $I=N$, their centers as $u_i=10e_i$, where $e_i$ is the $i$-th unit vector, and their radii as $R_i = 10 i$, $i=1,\dots,N$. We fix $\delta=10^{-7}$ in all computations in this section.
We sample 100 random objective function directions $\beta$ from the boundary of the unit ball (the coefficients are rounded to six decimal points).  We use Gurobi 9 on a 64-bit personal computer with Intel Core i7 CPU 2.60GHz processor (16 GB RAM) with the default settings except the following: i) In an attempt to limit the variations in the solution procedure, we set the  parameter \texttt{Method} to primal simplex method, which has worked the best in our preliminary experiments.  ii) The relative optimality tolerance parameter \texttt{MIPGap} is set to 0.01 for $N=32$. 

We summarize our computational results in Table~\ref{tab:optresults} by reporting the following metrics:
\begin{itemize}
\item Time (s): The computational time in seconds.
\item \# Iters: The total number of simplex iterations.
\item \# Nodes: The number of branch-and-bound nodes (for the IP Version).
\end{itemize}
We provide both the arithmetic average and the number of times each method \textit{strictly} wins in that category in the ``Averages'' and ``\# Wins'' columns, respectively.

\setlength{\tabcolsep}{5pt}
\begin{table}[H]
\centering
{
\footnotesize
\caption{Computational results for 100 random objective functions under different settings.}\label{tab:optresults}
\begin{tabular}{r|l|rrr|rrr|rrr|rrr}

 \multicolumn{ 1}{c }{ }           &   \multicolumn{ 1}{c }{ }              &                                    \multicolumn{ 6}{c|}{LP Version} &                                   \multicolumn{ 6}{c}{IP Version} \\

\cline{3-14}

 \multicolumn{ 1}{c }{ }           &   \multicolumn{ 1}{c| }{ }              &        \multicolumn{ 3}{c|}{Averages} &            \multicolumn{ 3}{c|}{\# Wins} &        \multicolumn{ 3}{c|}{Averages} &            \multicolumn{ 3}{c}{\# Wins} \\

\cline{3-14}

      \multicolumn{ 1}{c }{ }           &   \multicolumn{ 1}{c| }{ }           &          $\mathbb{R}$ &          $\mathbb{Q}$ &         $ \mathbb{Z} $ &          $\mathbb{R}$ &         $ \mathbb{Q}$ &        $  \mathbb{Z}$ &          $ \mathbb{R}$ &        $  \mathbb{Q}$ &        $  \mathbb{Z} $&       $    \mathbb{R} $&       $   \mathbb{Q} $ &         $ \mathbb{Z}$\\

\toprule

%
%
\parbox[t]{2mm}{\multirow{3}{*}{\rotatebox[origin=c]{90}{$N=8$}}}  
           &       Time (s) &      0.062 &      0.056 &      0.028 &          0 &          0 &        100 &      0.275 &      0.268 &      0.257 &         14 &         32 &         54 \\

           &   \#     Iters &    1611.12 &     1588.20 &    1166.71 &          0 &          0 &        100 &    6139.78 &     5899.80 &    5556.73 &          9 &         30 &         61 \\
           
&    \#    Nodes &          - &          - &          - &          - &          - &          - &      69.45 &      68.43 &      68.58 &         23 &         22 &         35 \\

\midrule
\parbox[t]{2mm}{\multirow{3}{*}{\rotatebox[origin=c]{90}{$N=16$}}}  

           &       Time (s) &      0.826 &      0.622 &      0.305 &          0 &          0 &        100 &      4.329 &      4.188 &      3.532 &         10 &         15 &         75 \\

           &    \#    Iters &    6273.46 &    5697.03 &    4935.45 &          0 &          0 &        100 &   48061.45 &   44991.74 &   33029.27 &          2 &          4 &         94 \\

&   \#     Nodes &          - &          - &          - &          - &          - &          - &     273.27 &     273.61 &     278.36 &         24 &         32 &         34 \\
\midrule
\parbox[t]{2mm}{\multirow{3}{*}{\rotatebox[origin=c]{90}{$N=32$}}} 

           &       Time (s) &     13.196 &     10.476 &       5.120 &          0 &          0 &        100 &    166.168 &    169.341 &     93.525 &          4 &          4 &         92 \\

           &       \#    Iters &   24580.03 &   22588.88 &   20029.49 &          0 &          0 &        100 &   883694.80 &   848579.10 &   220192.30 &          2 &          4 &         94 \\
&       \#     Nodes &             - &          - &          - &          - &          - &          - &    1549.25 &     1629.70 &     1588.10 &         30 &         27 &         39 \\

\bottomrule
\end{tabular}  
}
\end{table}
We will now discuss our observations from the computational experiments. 
When we consider the LP version, we clearly see that the $\mathbb{Z}$-Formulation is better than both $\mathbb{R}$-Formulation and $\mathbb{Q}$-Formulation as it  requires less number of simplex iterations and computational time in all instances.

For the IP version, the conclusions are not as clear-cut although there is a pattern  that the relative success of  the $\mathbb{Z}$-Formulation improves with the instance size.  The $\mathbb{Z}$-Formulation is the best  in terms of the computational time and number of iterations if we consider the averages and the number of wins. However, we note that there are some random directions in which the other formulations occasionally outperform the $\mathbb{Z}$-Formulation. Although the number of branch-and-bound nodes is quite similar for all the formulations, the $\mathbb{Z}$-Formulation has significantly smaller number of iterations per node. Finally, we remark that the performance of $\mathbb{R}$-Formulation and $\mathbb{Q}$-Formulation are quite similar in our experiments.

Overall, our computational results suggest that the $\mathbb{Z}$-Formulation proposed in this paper might have advantages compared the  $\mathbb{R}$-Formulation, at least in our stylized experiments. We would like to note that we have repeated the experiments with some modifications (choosing the same $R_i$ value for each ball, choosing the objective directions from a box, and choosing both the centers and radii of the balls randomly) and observed a similar behavior. 

\subsection{Discussion and Future Work}
We conclude our paper with a brief discussion on the applications discussed in this section. We  presented the first two applications from a theoretical perspective and did not provided any computational results due to the following reasons: For the integrality gap computation, it is necessary to compute (or accurately estimate)  the largest  subdeterminant of the integral matrix used in the outer-approximation. Unfortunately, this task is NP-Hard in general \cite{summa2014largest}. For the counting of integral vectors, the difficulty is the implementation of the integer projection algorithm \cite{barvinok2003proj}, which is quite tedious. To the best of our knowledge, the only implementation is carried out in \cite{koppe2008implementation} and tested on a set of toy examples. We believe that any computational attempt in these applications should exploit the structure of the outer-approximation, which we  leave  as future work.

For the computational application,  our aim was to provide some evidence about the usefulness of the rational outer-approximations against the irrational outer-approximations, which we were able to achieve. Nevertheless, we  also tried to directly solve  the optimization problem over the intersection of balls as an SOCP or MISOCP. We observed that such an approach is, in general, more efficient than solving via outer-approximation. We  leave the exploration of the rational outer-approximation as a competitive tool against directly solving the SOCP or MISOCP as a future research direction. 
}

\section*{Acknowledgments}
The author would like to thank Dr. Diego Moran for his comments on an earlier version of this paper,  {and two referees for their suggestions regarding Sections~\ref{sec:deltaGivenM} and \ref{sec:applications}.}

\bibliographystyle{plain}
\bibliography{references}

\end{document}